\documentclass[reqno, 11pt]{amsart}
\usepackage[utf8]{inputenc}
\usepackage[text={155mm,235mm},centering]{geometry}
\input diagxy
\input xy

\usepackage[active]{srcltx} 

\newtheorem{thm}{Theorem}[section]
\newtheorem{cor}[thm]{Corollary}
\newtheorem{lem}[thm]{Lemma}
\newtheorem{prop}[thm]{Proposition}

\theoremstyle{definition}
\newtheorem{defn}[thm]{Definition}
\theoremstyle{property}

\theoremstyle{remark}
\newtheorem{rem}[thm]{Remark}

\newtheorem{ex}[thm]{Example}
\numberwithin{equation}{section}

\usepackage[mathscr]{eucal}
\usepackage[all]{xy}
\usepackage{mathrsfs}
\usepackage{xypic}
\usepackage{amsfonts}
\usepackage{amsmath}
\usepackage{amsthm,bm}
\usepackage{amssymb,tikz-cd}
\usepackage{latexsym}
\usepackage{tabularx}
\usepackage{graphicx}
\usepackage{pict2e}
\usepackage{hyperref}
\usepackage{tikz}
\usepackage{textcomp}
\usepackage[scr=rsfs]{mathalfa}
\usepackage[all]{xy}
\usepackage{tikz-cd}
\usepackage{esint}
\definecolor{ceruleanblue}{rgb}{0.16, 0.32, 0.75}
\hypersetup{colorlinks=true,allcolors=ceruleanblue}
\allowdisplaybreaks

\begin{document}

\title[basic Kirwan injectivity]
 {Basic Kirwan injectivity and its applications}

\author{Yi Lin}

\address{Y. Lin \\
Department of Mathematical Sciences \\
Georgia Southern University \\
Statesboro, GA, 30460 USA}
\email{yilin@georgiasouthern.edu}
\author{Xiangdong Yang}

\address{X. Yang \\
Department of Mathematics \\
Lanzhou University \\
Lanzhou 730000, China}

\email{yangxd@lzu.edu.cn}

\subjclass[2010]{57S25; 57R91}

\keywords{transversely symplectic foliations, basic cohomology, Kirwan injectivity.}
\date{\today}



\begin{abstract}
Consider the Hamiltonian action of a torus on a transversely symplectic foliation that is also Riemannian.
When the transverse hard Lefschetz property is satisfied, we establish a foliated version of the Kirwan injectivity theorem, and use
it to study Hamiltonian torus actions on transversely K\"ahler foliations.
Among other things, we prove a foliated analogue of the Carrell--Liberman theorem.
As an application, this confirms a conjecture raised by Battaglia--Zaffran on the basic Hodge numbers of symplectic toric quasifolds.
Our methods also allow us to present a symplectic approach to the calculation of the Betti numbers of symplectic toric quasifolds as diffeological spaces.
\end{abstract}

\maketitle
\section{Introduction}

Reinhart \cite{Re59} introduced the notion of \emph{basic cohomology} as a cohomology theory for the leaf space of a foliation.
It turns out to be a very useful tool in the study of Riemannian foliations.
Killing foliation is an important class of Riemannian foliations, and is known to possess a type of \lq\lq \emph{internal}" symmetry given by the transverse action of their structural Lie algebras.
In order to study this important type of symmetries, Goertsches--T\"{o}ben \cite{GT10} proposed the notion of \emph{equivariant basic cohomology},
and established a Borel type localization theorem for the transverse actions of structural Lie algebras on Killing foliations.
Lin--Sjamaar  \cite{LS18} generalized their result to the transverse isometric action of an arbitrary Lie algebra on a general Riemannian foliation.

In a different direction,  Lin--Sjamaar \cite{LS17} considered Hamiltonian action of a compact Lie group on a transversely symplectic foliation.
They discovered that when the action is \emph{clean},  components of a moment map must be Morse--Bott functions, and extended Atiyah--Guillemin--Sternberg--Kirwan convexity theorem to clean Hamiltonian actions.
Lin--Yang \cite{LY17} studied Hamiltonian actions on a transversely symplectic foliation from Hodge theoretic viewpoint, and established in this setup the equivariant formality result for the equivariant basic cohomology.

In this paper, we set ourselves a modest goal of applying the symplectic methods developed in \cite{LS17}, \cite{LY17} and \cite{LS18} to derive some interesting applications in the geometry and topology of foliations.
We first present a quick proof of a foliated version of the Kirwan injectivity theorem for transversely symplectic foliations (Theorem \ref{thm1}),
then move on to use it to study Hamiltonian torus actions on transversely symplectic and K\"{a}hler foliations.
When the action preserves the transversely holomorphic structure, we obtain a foliation analogue of the Carrell--Lieberman theorem \cite{CL73} and establish the vanishing of certain basic Hodge numbers (Theorem \ref{thm2}).
When the Hamiltonian action is clean and the fixed leaf set consists entirely of isolated closed leaves, we prove that the odd dimensional basic Betti number must vanish, and that the basic Euler characteristic number is given by the total number of the vertices in the moment polytope.

Out of any non-rational simple polytope,  Prato \cite[Section 3]{Pr01} proposed a variation of the Delzant construction and constructed a generalized version of a symplectic toric manifold, called a symplectic toric quasifold.
Subsequently, Battaglia--Prato \cite{BP01} introduced complex toric quasifolds for simplicial polytopal fans, and clarified the connections between the symplectic and complex approaches to toric quasifolds.
Since then, much research has been done to understand the cohomology of a toric quasifold. Battaglia \cite{Ba11} first determined the betti numbers of the De Rham cohomology of a toric quasifold.
Battaglia--Zaffran \cite{BZ15} noted that complex toric quasifolds could be realized as the leaf spaces of holomorphic foliations on compact LVMB complex manifolds, and extended Battaglia's result in the framework of basic Betti numbers for foliations.
In a follow up work \cite{BZ17}, Battaglia--Zaffran also explained in details the relation between LVM manifolds and complex and symplectic toric quasifolds. These methods and results were further developed in \cite{Ish18} and \cite{IKP18}.
Among other things, Ishida--Krutowiski--Panov \cite{IKP18} computed the basic cohomology ring of any complex moment-angle manifold.

From our viewpoint, Prato's toric quasifolds provide a rich and interesting class of Hamiltonian torus actions on transversely K\"ahler foliations to which our general results apply very well. Indeed, our foliation analogue of the Carrell--Liberman theorem provides a positive answer to a conjecture raised by Battaglia--Zaffran \cite{BZ15},  which asserts that the basic Hodge numbers of toric quasifolds are concentrated on the main diagonal.
As an aside, we also present a purely symplectic method of computing the betti numbers of a symplectic toric quasifold as a diffeological space, c.f. \cite{I13}.
Our method appears to be new even for symplectic toric manifolds (cf. \cite[Theorem I.3.6]{Ca03}). It relies on the extensions of two important facts in equivariant symplectic geometry to the realm of Hamiltonian actions on transversely symplectic foliations: the uniqueness of the minimal critical submanifold of a component of a moment map, and the slice theorem for Hamiltonian torus actions.
It may also be noteworthy that our main results apply to a wider class of transversely K\"ahler foliations which may not admit a large symmetry from a torus of greatest possible dimension.

The rest of this paper is organized as follows.
Section \ref{sec2} recalls the definitions of transverse geometric structures on foliations. Section \ref{sec3} reviews the equivariant basic cohomology and establishes the Kirwan injectivity theorem.
Section \ref{sec4} proves some useful facts on clean Hamiltonian torus actions. Section \ref{general-theory} proves an analogue of the Carrell--Liberman theorem in a foliated setting, as well as a result on the basic Betti numbers and basic Euler characteristic numbers of transversely symplectic foliations with Hamiltonian torus actions.
Section \ref{quasifolds} applies the main results to determine the basic Hodge numbers and basic Betti numbers of  symplectic toric quasifolds.


\section{Transverse geometric structures on foliations}\label{sec2}

Let $\mathcal{F}$ be a foliation on a smooth manifold $M$, and let $T\mathcal{F}$ be the tangent bundle of the foliation.
Throughout this paper we denote by
$\mathfrak{X}(\mathcal{F})\subset\mathfrak{X}(M)$ the subspace of vector fields tangent to the leaves of $\mathcal{F}$.
We say that a vector field $X\in\mathfrak{X}(M)$ is \emph{foliate}, if $[X,Y]\in\mathfrak{X}(\mathcal{F})$ for all $Y\in \mathfrak{X}(\mathcal{F})$.
We will denote by
$\mathfrak{R}(\mathcal{F})$ the space of foliate vector fields on $(M,\mathcal{F})$.
Clearly we have that
$\mathfrak{X}(\mathcal{F}) \subset \mathfrak{R}(\mathcal{F})$.
In this context,  a \emph{transverse vector field} is an equivalence class in the quotient space $\mathfrak{R}(\mathcal{F})/\mathfrak{X}(\mathcal{F})$. The space of transverse vector fields, denoted by $\mathfrak{X}(M, \mathcal{F})$, forms a Lie algebra with a Lie bracket inherited from the natural one on $\mathfrak{R}(\mathcal{F})$.
The space of \emph{basic forms} on $M$ is defined to be
$$
\Omega(M,\mathcal{F})=
\bigl\{\alpha\in\Omega(M)\,|\,
\iota(X)\alpha=\mathcal{L}(X)\alpha=0,\,
\mathrm{for}\,
\,X\in\mathfrak{X}(\mathcal{F})\bigr\}.
$$
Since the exterior differential operator $d$ preserves basic forms, we obtain a sub-complex  $\{\Omega^{*}(M,\mathcal{F}),d\}$ of the usual de Rham complex, called the \emph{basic de Rham complex}.
The associated cohomology $H^{*}(M,\mathcal{F})$
is called the \emph{basic cohomology}.

\begin{defn}
A \emph{transverse Riemannian metric} on a foliation $(M,\mathcal{F})$ is a Riemannian metric $g$ on the normal bundle $N=TM/T\mathcal{F}$ of the foliation, such that $\mathcal{L}(X)g=0$,
for any $X\in\mathfrak{X}(\mathcal{F})$.
We say that $\mathcal{F}$ is a Riemannian foliation if there exists a transverse Riemannian metric on $(M,\mathcal{F})$.
Let $q$ be the co-dimension of a Riemannian foliation $\mathcal{F}$ on a compact manifold $M$.
We say that $\mathcal{F}$ is \emph{taut} if $H^{q}(M,\mathcal{F})=\mathbb{R}$, .
\end{defn}

\begin{defn}\label{foliate-bundle} Let $(M, \mathcal{F})$ be a foliated manifold. A $k$ dimensional vector bundle $\pi: E\rightarrow M$ is said to be foliated, if
it is equipped with a foliation $\tilde{\mathcal{F}}$, of the same dimension as $\mathcal{F}$,  whose leaves are transversal to the fibers of $\pi$ and mapped by $\pi$ to those of $\mathcal{F}$. When this is the case, a section $s: M\rightarrow E$ is said to be a \emph{foliated section}, if it maps a leaf of $\mathcal{F}$ in $M$ into a leaf of $\tilde{\mathcal{F}}$ in $E$. \end{defn}

It is well-known that the normal bundle $N$ of the foliation $\mathcal{F}$ is an important example of a foliated vector bundle, which is naturally equipped with a foliation $\tilde{\mathcal{F}}$ induced by the Bott connection. As a result, the dual vector bundle $N^*$ of $N$, and the determinant line bundle $L=\wedge^q N^*$ of $N^*$, are also foliated vector bundles over $M$, where $q=\text{codim} (\mathcal{F})$. Clearly, there is a natural one-to-one correspondence between the space of foliated sections of $L^*$ and the space of basic forms of top degrees on $M$. We say that $(M, \mathcal{F})$ is \emph{transversely orientable}, if there is a nowhere vanishing foliated section of $L$, or equivalently, there is a top degree no-where vanishing basic form on $M$.
A \emph{transverse orientation} is an equivalence class of the following equivalence relations on the space of  top degree basic forms: $\alpha\sim \beta, \text{ if}\,  \alpha=f \beta$ for some positive basic function $f$ on $M$.

If $(M, \mathcal{F})$ is equipped with a transverse Riemannian metric $g$, then $g$ induces a point-wise Riemannian norm on the space of basic forms of top degrees. In this context, a \emph{transverse volume form} is a no-where vanishing basic form of top degree whose induced point-wise Riemannian norm equals one everywhere.  A transverse volume form for a transverse Riemannian metric $g$ on $(M,\mathcal{F})$ exists if and only if $(M, \mathcal{F})$ is transversely orientable. Moreover, when
$(M, \mathcal{F})$ is transversely orientable, then a given transverse orientation class on $(M, \mathcal{F})$ contains  a unique transverse volume.

Let $(M,\mathcal{F})$ be a Riemannian foliation with a transverse metric $g$, and let $\overline{X}$ be a transverse vector field.
Define
$\mathcal{L}(\overline{X})g=\mathcal{L}(X)g$,
where $X$ is a foliate vector field that represents $\overline{X}$.
It is straightforward to check that this definition does not depend on the choice of a foliate vector representing $\overline{X}$.
A transverse vector field $\overline{X}$ is said to be \emph{transversely Killing} if $\mathcal{L}(\overline{X})g=0$.
Suppose that both $\overline{X}$ and $\overline{Y}$ are transversely Killing.
Then it follows easily from the Cartan identities that $[\overline{X},\overline{Y}]$ is also transversely Killing.
In other words, the space of transversely Killing vector fields, which we denote by $\text{Iso}(M, \mathcal{F})$, forms a Lie subalgebra of $\mathfrak{X}(M, \mathcal{F})$.

\begin{defn}
A  \emph{transverse almost complex structure} $\mathcal{J}$ on $(M,\mathcal{F})$ is an almost complex structure
$\mathcal{J}: TM/T\mathcal{F}\rightarrow TM/T\mathcal{F}$
such that
$\mathcal{L}(X)\mathcal{J}=0$,
for any $X\in \mathfrak{X}(\mathcal{F})$.
A  transverse almost complex structure $\mathcal{J}$ on $(M,\mathcal{F})$ is said to be \emph{integrable}, if
for every $p\in M$, there exists an open neighborhood $U$ of $p$, such that for any two transverse vector fields $\overline{X}$ and $\overline{Y}$  on $U$ with respect to the foliation $\mathcal{F}\vert_U$, the Nijenhaus tensor
$$
N_{\mathcal{J}}(\overline{X},\overline{Y})=
[\mathcal{J}\overline{X},\mathcal{J}\overline{Y}]
-\mathcal{J}[\mathcal{J}\overline{X},\overline{Y}]
-\mathcal{J}[\overline{X}, \mathcal{J}\overline{Y}]-[\overline{X},\overline{Y}]
$$
vanishes.
\end{defn}

An integrable transverse almost complex structure is also called a \emph{transverse complex structure}.
The foliation $\mathcal{F}$ is said to be \emph{transversely holomorphic} if there is a transverse complex structure $\mathcal{J}$ on $(M, \mathcal{F})$.

\begin{defn}
Let $\mathcal{F}$ be a foliation on a smooth manifold $M$.
We say that $\mathcal{F}$ is a \emph{transversely symplectic foliation},
if there exists a closed 2-form $\omega$, called a \emph{transversely symplectic form},
such that for each $x\in M$, the kernel of $\omega_x$ coincides with $T_x\mathcal{F}$.
\end{defn}

\begin{defn}\label{kahler-foliation}
A \emph{transversely K\"{a}hler structure} on $(M,\mathcal{F})$ consists of a transverse complex structure $\mathcal{J}$ and a transverse Riemannian metric $g$,  such that the tensor field $\omega$ defined by $\omega(X,Y)=g(X,\mathcal{J}Y)$ is transversely symplectic when considered as a 2-form on $M$ given by the injection $\wedge^{2}N^*\rightarrow\wedge^{2}T^{*}M$, where $N^*$ is the dual of the normal bundle of the foliation.
The $2$-form $\omega$  will be called a \emph{transversely K\"{a}hler form}.
We say that $\mathcal{F}$ is a transversely K\"ahler foliation if there exists a transversely K\"ahler structure on $(M,\mathcal{J})$.
\end{defn}
A submanifold $X$ of a foliated manifold $(M, \mathcal{F})$ is said to be \emph{saturated}, if any leaves of $\mathcal{F}$ that intersect $X$ non-trivially must be contained in $X$.
Let $(g, \mathcal{J})$ be a transversely K\"ahler structure on $(M, \mathcal{F})$ and $X$ a saturated submanifold of $(M, \mathcal{F})$.
We say that $X$ is a \emph{transversely K\"ahler submanifold} if $TX/T(\mathcal{F}\vert_{X})$, the normal bundle of the restricted foliation $\mathcal{F}\vert_X$, is invariant under the transverse complex structure $\mathcal{J}$.

\begin{lem}\label{kahler-form} Let $(M, \mathcal{F})$ be a transversely K\"ahler foliation, let $X$ be a transversely K\"ahler  submanifold of $M$ such that the restricted foliation $\mathcal{F}\vert_X$ has a co-dimension $2k$ in $X$,  and let $\omega$ be the transversely K\"ahler form on $(M, \mathcal{F})$ as given in Definition \ref{kahler-foliation}.
Then $\omega_X=\omega\vert_X$ is a transversely K\"ahler form on $X$. Moreover, if we denote by $\nu_X$ the transverse volume form on $X$ determined by the transverse orientation of $\omega_X^k$, then
\[ \omega_{X}^k=\frac{1}{k!}\nu_X.\]
\end{lem}
\begin{proof}
The standard argument in K\"ahler geometry (for the case of a point foliation) extends to the present situation without any essential changes.
We refer to \cite[Page 31]{GH78} for details.
\end{proof}

Applying the basic Hodge theory developed in \cite{KA90}, the following foliated version of the $\bar{\partial}\partial$-lemma
was shown in \cite[Lemma 1]{CW91}.
\begin{thm}\label{basic-ddbar-lemma}
Suppose that $\mathcal{F}$ is a taut transversely K\"ahler foliation on a compact manifold $M$.
Then on the space of basic forms $\Omega(M,\mathcal{F})$ the following $\bar{\partial}\partial$-lemma holds.
\[\ker\,\bar{\partial} \cap \mathrm{im}\,\partial=\mathrm{im}\, \bar{\partial}\cap \ker\,\partial=\mathrm{im}\,\bar{\partial}\partial.\]
\end{thm}

\section{Kirwan injectivity theorem for the equivariant basic cohomology}\label{sec3}

In this section, we prove a foliated version of the Kirwan injectivity theorem.
We begin with a review of the notion of transverse Lie algebra actions, and the associated notion of equivariant basic cohomology.

\begin{defn}
A \emph{transverse action} of a Lie algebra $\mathfrak{g}$ on a foliated manifold $(M,\mathcal{F})$ is a Lie algebra homomorphism
\begin{equation}\label{equ1.1}
\mathfrak{g}\longrightarrow\mathfrak{X}(M, \mathcal{F}).
\end{equation}
A transverse action of $\mathfrak{g}$ on $(M,\mathcal{F})$ is said to be isometric, if the image of the map (\ref{equ1.1}) lies inside $\text{Iso}(M, \mathcal{F})$.
\end{defn}

Suppose that there is a transverse action of a Lie algebra $\mathfrak{g}$ on a foliated manifold $(M,\mathcal{F})$.
For all $\xi\in \mathfrak{g}$, we will denote by $\overline{\xi}_M$ the transverse vector that is the image of $\xi$ under (\ref{equ1.1}), and by $\xi_M$ the foliate vector that represents $\overline{\xi}_M$.
For $\alpha \in \Omega(M,\mathcal{F})$, define
\[ \iota(\xi)\alpha =\iota(\xi_M)\alpha,\,\,\,\,
\mathcal{L}(\xi)\alpha=\mathcal{L}(\xi_M)\alpha.\]
Since $\alpha$ is basic, the contraction and Lie derivative operations defined above do not depend on the choices of representatives of the transverse vector field $\overline{\xi}_M$.
Goertsches--T\"{o}ben \cite[Proposition 3.12]{GT10} observed that they obey the usual rules of Cartan's differential calculus, namely $[\mathcal{L}(\xi),\mathcal{L}(\eta)]
=\mathcal{L}([\xi,\eta])$ etc.
To put it another way, a transverse $\mathfrak{g}$-action equips the basic de Rham complex $\Omega(M,\mathcal{F})$ with the structure of a $\mathfrak{g}^{\star}$-algebra in the sense of  \cite[Chapter 2]{GS99}.
Therefore there is a well-defined Cartan model
of the $\mathfrak{g}^{\star}$-algebra $\Omega(M,\mathcal{F})$ given by
\[
\Omega_{\mathfrak{g}}(M,\mathcal{F}):=
[S\mathfrak{g}^* \otimes \Omega(M,\mathcal{F})]^{\mathfrak{g}}.
\]
An element of $\Omega_{\mathfrak{g}}(M,\mathcal{F})$ can be naturally identified with an equivariant polynomial map from
$\mathfrak{g}$ to $\Omega(M,\mathcal{F})$, and is called an \emph{equivariant basic differential form}.

The equivariant basic Cartan complex has a bigrading given by
\[ \Omega_{\mathfrak{g}}^{i,j}(M,\mathcal{F})=
[S^i\mathfrak{g}^*\otimes \Omega^{j-i}(M,\mathcal{F})]^{\mathfrak{g}};
\]
moreover, it is equipped with the vertical differential $1\otimes d$, which we abbreviate to $d$, and the horizontal differential
$d'$, which is defined by
\[
(d'\alpha)(\xi)=-\iota(\xi)\alpha(\xi), \,\,\,
\forall\, \xi\in\mathfrak{g}.
\]
As a single complex, $\Omega_{\mathfrak{g}}(M,\mathcal{F})$ has a grading given by
\[
\Omega^k_{\mathfrak{g}}(M,\mathcal{F})=
\displaystyle \bigoplus_{i+j=k}
\Omega_{\mathfrak{g}}^{i,j}(M,\mathcal{F}),
\]
and a total differential $d_{\mathfrak{g}}=d+d'$, which is called the equivariant exterior differential.
The equivariant basic de Rham cohomology $H_{\mathfrak{g}}(M,\mathcal{F})$ of the transverse $\mathfrak{g}$-action on $(M,\mathcal{F})$ is defined to be the total cohomology of the Cartan complex $\{\Omega_{\mathfrak{g}}(M,\mathcal{F}), d_{\mathfrak{g}}\}$.

Now let $G$ be a compact connected Lie group with Lie algebra $\mathfrak{g}$.
A transverse action of $G$ on a foliated manifold $(M,\mathcal{F})$ is simply a transverse action of its Lie algebra $\mathfrak{g}$ on $(M,\mathcal{F})$.
The equivariant basic cohomology  of a transverse $G$-action on $(M,\mathcal{F})$ is defined to be
\[
H_G(M,\mathcal{F}):= H_{\mathfrak{g}}(M,\mathcal{F}).
\]

We say that the action of a Lie group $G$ on a foliated manifold $(M,\mathcal{F})$ is \emph{foliate}, if the action preserves the foliation structure.
Suppose that there is a foliate $G$-action on a foliated manifold $(M,\mathcal{F})$.
Then we have the following commutative diagram
$$
 \xymatrix{
  \mathfrak{g} \ar[dr]_{} \ar[r]^{}
                & \mathfrak{R}(\mathcal{F})
                \ar[d]^{\mathrm{pr}}  \\
                & \mathfrak{X}(M, \mathcal{F}).             }
$$
Here the horizontal map is induced by the infinitesimal action of $\mathfrak{g}$ on $M$, and the vertical map is the natural projection.
Thus a foliate action of $G$ naturally induces a transverse action of $G$.
As transverse vector fields are not genuine vector fields on $M$, the converse may not be true.

\begin{defn} (\cite{LS18})
Consider the action of a Lie group $G$ with Lie algebra $\mathfrak{g}$ on a foliated manifold $M$ with a transversely symplectic foliation $(\mathcal{F},\omega)$.
We say that the action of $G$ is \emph{Hamiltonian},
if there exists a $G$-equivariant map $\Phi:M\rightarrow\mathfrak{g}^{*}$,
called the moment map, such that
\begin{equation*}\label{ham-action}
\iota(\xi_{M})\omega=d\langle \Phi,\xi\rangle,\,\,\,\mathrm{for}\,\,\mathrm{all}\,\,\xi\in\mathfrak{g}.
\end{equation*}
Here $\xi_{M}$ is the fundamental vector field on $M$ generated by $\xi$, and $\langle\cdot,\cdot\rangle$ denotes the dual pairing between $\mathfrak{g}^*$ and $\mathfrak{g}$.
It is easy to check that a Hamiltonian action must automatically be foliate.
\end{defn}

\begin{defn}\label{def5.3}
A foliate $G$-action on $(M,\mathcal{F},\mathcal{J})$ is \emph{holomorphic},
if the induced $G$-action on the normal bundle of the foliation $TM/T\mathcal{F}$ preserves the transverse complex structure $\mathcal{J}$.
\end{defn}

We recall the following equivariant formality result proved in \cite[Theorem 1.1]{LY17}.

\begin{thm}[Equivariant Formality {\cite[Theorem 1.1]{LY17}}]\label{eq-formality}
Consider the Hamiltonian action of a compact group $G$ on a transversely symplectic foliation $(\mathcal{F}, \omega)$ over a compact manifold $M$.
Suppose that $(M,\mathcal{F},\omega)$ satisfies the transverse hard Lefschetz property.
Then there is a canonical $(S\mathfrak{g}^{*})^G$-module isomorphism from the equivariant basic cohomology $H_{G}(M,\mathcal{F})$ to
$
\left(S\mathfrak{g}^{*}\right)^G\otimes H(M,\mathcal{F}).
$
\end{thm}

As a direct application of Theorem \ref{eq-formality} and the Borel localization result  in \cite{LS18}, we establish the following foliated version of the Kirwan injectivity theorem in symplectic geometry.

\begin{thm}\label{thm1}
Let  $(M, \mathcal{F})$ be a transversely symplectic foliation on a closed manifold $M$ that is also Riemannian.
Suppose that $(M, \mathcal{F})$ satisfies the transverse hard Lefschetz property, that there is a Hamiltonian action of a compact torus $G$ on $M$.
Let $X$ be the fixed-leaf set of $M$
and $i: X\hookrightarrow M$ the inclusion map.
Then the localization homomorphism in equivariant basic cohomology
\[ i^*: H_G(M,\mathcal{F})\longrightarrow H_G(X,\mathcal{F})\]
is injective.
\end{thm}
\begin{proof}
By assumption, there is a transverse Riemannian metric $g$ on $(M,\mathcal{F})$.
By averaging over the compact torus $G$ we may assume that
the induced transverse action of $G$ is isometric.
Then by \cite[Theorem 4.7]{LS18} connected components of the fixed-leaf set $X$ are $\overline{\mathcal{F}}$-saturated submanifolds of $M$ that are invariant under the action of $G$.
By \cite[Theorem 5.1]{LS18}, the kernel of the restriction homomorphism
\begin{equation} \label{restriction-homo}
i^*: H_{G}(M,\mathcal{F})\rightarrow H_G(X, \mathcal{F})
\end{equation} is a $S\mathfrak{g}^*$-torsion submodule.
However, by Theorem \ref{eq-formality} $H_G(M,\mathcal{F})$ is a free $S\mathfrak{g}^*$-module.
Therefore the map (\ref{restriction-homo}) must be injective.
\end{proof}

According to El Kacimi-Alaoui \cite[Section 3.4.7]{KA90}, any taut transversely K\"{a}hler foliation on closed manifolds must satisfy the transverse hard Lefschetz property.
As an immediate consequence of Theorem \ref{thm1} we have the following result.
\begin{cor}\label{kirwan-injectivity}
Let $\mathcal{F}$ be a taut transversely K\"{a}hler foliation on a closed manifold $M$.
Suppose that there is a Hamiltonian action of a compact torus $G$ on $(M,\mathcal{F})$ that is also holomorphic.
Then the restriction homomorphism
\[
i^*: H_G(M,\mathcal{F})\rightarrow H_G(X,\mathcal{F})
\]
is injective, where $X$ is the fixed-leaf set.
\end{cor}

\section{Clean Hamiltonian torus actions on transversely symplectic foliations}\label{sec4}

In this section we collect some useful facts on the image of a moment map of a clean Hamiltonian torus action, which we need later in this paper.
We recall the definition of \emph{clean} group actions on foliated manifolds.

\begin{defn}(\cite[Section 2.6]{LS17})\label{clean-action}
Consider the foliate action of a Lie group $G$ on a foliated manifold $(M, \mathcal{F})$.
We say that the $G$-action on $M$ is clean,
if there exists an immersed connected normal Lie subgroup $N$ of $G$, called the \emph{null subgroup}, such that
$$
T_{x}(N\cdot x)=T_{x}(G\cdot x)\cap T_{x}\mathcal{F},
\,\,\mathrm{for}\,\,\mathrm{all}\,\,x\in M.
$$
\end{defn}

\begin{defn}
Suppose that $\mathcal{F}$ is a transversely symplectic foliation on a manifold $M$, and that there is a clean foliate action of a compact torus $T$ with Lie algebra $\mathfrak{t}$ on $(M,\mathcal{F})$.
For all $x\in M$, set
\[
\mathfrak{t}_{\bar{x}}=\{ \zeta\in \mathfrak{t}\,\vert\,\zeta_{M}(x) \in T_x\mathcal{F}\}.
\]
We say that $\xi \in \mathfrak{t}$ is \emph{generic}, if
\[
\xi\notin \bigcup_{\mathfrak{t}_{\bar{x}}\neq \mathfrak{t}} \mathfrak{t}_{\bar{x}}.
\]
\end{defn}

\begin{rem}
Suppose that $N$ is the null subgroup of $T$, and that $\mathfrak{n}=\text{Lie}(N)$.
Since the action of $T$ is assumed to be clean, it is straightforward to check that $\mathfrak{t}_{\bar{x}}=\mathfrak{t}_x+\mathfrak{n}$. If $M$ is compact, then there are only finitely many distinct isotropy subalgebras $\mathfrak{t}_x$, and so only finitely many Lie subalgebras $\mathfrak{t}_{\bar{x}}$.
\end{rem}

Throughout the rest of this section, we assume that there is an \emph{effective} and \emph{clean} Hamiltonian action of a compact torus $T$ on a transversely symplectic foliation $\mathcal{F}$ over a compact manifold $M$, that the null group is $N$ with a Lie algebra $\mathfrak{n}$, and that
$\Phi: M\rightarrow \mathfrak{t}^*$ is a moment map, where $\mathfrak{t}=\text{Lie}(T)$.

\begin{lem}\label{generic-choice}
For any generic element $\xi \in
\mathfrak{t}^*$, the critical subset $\mathrm{Crit}\,(\Phi^{\xi})$ of $\Phi^{\xi}$ coincides with $X$, the set of fixed leaves under the action of $T$.
\end{lem}

\begin{proof}
It suffices to show that
$\text{Crit}\,(\Phi^{\xi})\subset X$.
Suppose that $x \in M$ is a critical point of $\Phi^{\xi}$.
Then it follows from the Hamiltonian equation $\iota(\xi_M)\omega=d\Phi^{\xi}$ that $\xi_M$
is tangent to the leaf at $x$. This shows that $\xi_M\in \mathfrak{t}_{\overline{x}}$.
Since $\xi$ is generic, we must have that $\mathfrak{t}_{\bar{x}}=\mathfrak{t}$.
Thus for all $ \zeta\in\mathfrak{t}$, the vector field $\zeta_M$ must be tangent to the leaf at $x$.
It follows that the leaf through $x$ is invariant under the action of $T$.
This completes the proof of Lemma \ref{generic-choice}.
\end{proof}

\begin{lem} \label{faces}
The image $\Delta=\Phi(M)$ of the moment map
is a convex polytope, called the \emph{moment polytope}, contained in the annihilator of $\mathfrak{n}$ in $\mathfrak{g}^*$.
Moreover, for any face $F\subset \Delta$, the set $\Phi^{-1}(F)$ is a connected $\mathcal{F}$-saturated $T$-invariant submanifold of $M$ to which the restriction of $\mathcal{F}$ is transversely symplectic.
\end{lem}

\begin{proof}
The first assertion was shown in \cite[Theorem 3.5.1]{LS17}.
Now let
\[
\mathfrak{h}_F=
\{\xi\in \mathfrak{t}\,\vert\, \forall\, x\in F, \, \langle x, \xi\rangle=\text{max}_{y\in \Delta}\langle y, \xi\rangle ,\, \text{or}\,\langle x, \xi\rangle=\text{min}_{y\in \Delta}\langle y, \xi\rangle
\},
\] where $\langle\cdot,\cdot\rangle$ denotes the dual pairing between $\mathfrak{t}^*$ and $\mathfrak{t}$.
Then it is easy to see that $ \mathfrak{h}_F$ is a Lie subalgebra of $\text{Lie}(T)$ that contains $\mathfrak{n}$,
and that $\Phi^{-1}(F)$ is a disjoint union of connected components of
\[M^{[\mathfrak{h}_F]}=
\{x \in M\,\vert\, \xi_M(x)\in T_x\mathcal{F},\,\, \mathrm{for}\,\,\,\mathrm{all}\,\, \xi\in \mathfrak{h}_F\},\]
where $\xi_M$ denotes the fundamental vector field on $M$ generated by $\xi\in \mathfrak{h}_F$.
It follows from \cite[Proposition 3.4.4]{LS17} that each connected component of $\Phi^{-1}(F)$ must be a $\mathcal{F}$-saturated submanifold of $M$ to which the restriction of $\mathcal{F}$ is transversely symplectic.
To finish the proof of Lemma \ref{faces}, it suffices to show that $\Phi^{-1}(F)$ is connected.
By mathematical induction on the co-dimension of $F$ in $\Delta$, we may assume that $F$ is a facet of $\Delta$.
In this case,
$\mathfrak{h}_F/\mathfrak{n}$ is one dimensional. Choose a vector $\zeta\in \mathfrak{h}_F$ such that its image under the quotient map $\mathfrak{h}_F\rightarrow \mathfrak{h}_F/\mathfrak{n}$ is not zero.
Then $\Phi^{-1}(F)$ is precisely the maximum or minimal critical submanifold of $\Phi^{\zeta}$.
Note that by \cite[Theorem 3.4.5]{LS17} $\Phi^{\zeta}$ is a Morse--Bott function with even indices.
It now follows from a classical result of Atiyah \cite[Lemma 2.1]{At82} that $\Phi^{-1}(F)$ must be connected.
\end{proof}

Let $\xi\in\mathfrak{t}$ be a generic element. Then by Lemma \ref{generic-choice}, the critical subset of $\Phi^{\xi}$ coincides with the set of fixed leaves under the action of the full torus $G$.
Therefore it follows from  \cite[Theorem 3.5.1]{LS17} that there is a one-to-one correspondence between the connected components of critical submanifold of the Morse--Bott function $\Phi^{\xi}$ and the vertices of the moment polytope $\Delta=\Phi(M)$.

\begin{defn} \label{vertex-index}
Fix a generic element $\xi\in \mathfrak{t}$.
We will say that a vertex $\lambda\in \Delta$ has \emph{index} $l$ relative to the Morse--Bott function $\Phi^{\xi}$,  if the corresponding critical submanifold $\Phi^{-1}(\lambda)$ has Morse index $l$.
The index of the vertex $\lambda$ will be denoted by $\text{ind}_{\xi}(\lambda)$.
\end{defn}

\begin{lem} \label{min-vertex}
Let $\xi\in \mathfrak{t}$ be a generic element, and let $F$ be a non-empty face of the moment polytope $\Delta$.
Then there is a unique vertex of $F$, denoted by $\lambda_F$, such that
\begin{equation}\label{min} \mathrm{ind}_{\xi}(\lambda_F)<\mathrm{ind}_{\xi}(\nu),
\end{equation}
for all vertices $\nu\in F$ with $\nu\neq\lambda_F$.
The point $\lambda_F$ will be called the vertex of $F$ with the minimal index.
\end{lem}

\begin{proof}
Note that by Lemma \ref{faces} the set $Y=\Phi^{-1}(F)$ is a connected $\mathcal{F}$-saturated $T$-invariant submanifold of $M$ to which the restriction of $\mathcal{F}$ is transversely symplectic.
Clearly, the action of $T$ on $(Y,
\mathcal{F}\vert_{Y})$ is also clean and Hamiltonian. Thus by \cite[Theorem 3.4.5]{LS17}, the function $\Phi^{\xi}\vert_{Y}$ is also a Morse--Bott function with even indexes.
As a result, the restriction of $\Phi^{\xi}$ to $Y$ must have a unique local minimum on $Y$, see \cite[Lemma 2.1]{At82}.
It follows from Lemma \ref{generic-choice} and \cite[Theorem 3.5.1]{LS17} that this local minimum coincides with
$\Phi^{-1}(\lambda_F)$ for some vertex $\lambda_F \in F$. Finally, (\ref{min}) follows from the uniqueness of the local minimum.
\end{proof}

\section{Basic Hodge numbers and basic Betti numbers}\label{general-theory}

\subsection{A foliated Carrell--Lieberman-type theorem}

Let $\mathcal{J}$ be a transversely holomorphic structure on a foliation $(M,\mathcal{F})$.
Then $\mathcal{J}$ induces a direct sum decomposition of the complex of basic forms
$$
\Omega(M,\mathcal{F}) =\bigoplus_{p,q\geq0} \Omega^{p,q}(M,\mathcal{F}),
$$
where by definition a basic form $\alpha \in \Omega^{p,q}(M,\mathcal{F})$ if and only if $\mathcal{J}\alpha=(\sqrt{-1})^{p-q}\alpha$.
Let $\alpha$ be a basic form of $(p,q)$-type.
Define $\bar{\partial}\alpha$ to be the component of $d\alpha$ that lies in $\Omega^{p,q+1}(M,\mathcal{F})$, and $\partial\alpha$ the component of $d\alpha$ that lies in $\Omega^{p+1,q}(M,\mathcal{F})$.
Then the exterior differential $d$ naturally splits as $d=\bar{\partial}+\partial$.

Now suppose that there is a foliate action of a compact Lie group $G$ with Lie algebra $\mathfrak{g}$ on $(M,\mathcal{F})$ which preserves the transversely holomorphic structure $\mathcal{J}$.
Since the action of $G$ is holomorphic,  $\Omega^{p,q}(M,\mathcal{F})$ is a $G$-module for all $(p, q)$.
Thus the space
\begin{equation}\label{dolb-bigrading} \Omega_{\mathfrak{g}}^{p,q}(M):=
\bigoplus_{p'+i=p\atop{q'+i=q}}\bigl(S^i(\mathfrak{g}^*)\otimes \Omega^{p'q'}(M,\mathcal{F})\bigr)^{\mathfrak{g}}
\end{equation}
is well defined for all $(p, q)$.

For any $\xi\in\mathfrak{g}$, denote by $\overline{\xi}^{1,0}_M$ and $\overline{\xi}^{0,1}_M$ respectively the $(1,0)$ and $(0,1)$ components of the transverse vector field on $M$ induced by $\xi$.
Then on the basic Cartan complex $\Omega_{\mathfrak{g}}(M,\mathcal{F})$ the operator $d'$ splits as $d'=d^{'1,0}+d^{'0,1}$, where $$(d^{'1,0}\alpha)(\xi)= \iota(\overline{\xi}^{1,0}_M)(\alpha(\xi))
\,\,\,\mathrm{and}\,\,\,
(d^{'0,1}\alpha)(\xi)=
\iota(\overline{\xi}^{0,1}_M)(\alpha(\xi)),
$$
for all $ \alpha\in \Omega_{\mathfrak{g}}(M,\mathcal{F})$ and for all $\xi\in \mathfrak{g}$.
Also note that on the space of equivariant basic forms $\Omega_{\mathfrak{g}}(M,\mathcal{F})$ the operator
$1\otimes d$ splits as
$1\otimes d=1\otimes \bar{\partial}+1\otimes \partial$.
For brevity, we will also abbreviate
$1\otimes \bar{\partial}$ to $\bar{\partial}$ and $1\otimes \partial$ to $\partial$.
Thus the equivariant exterior differential $d_{\mathfrak{g}}$ splits as $d_{\mathfrak{g}}=
\bar{\partial}_{\mathfrak{g}}+\partial_{\mathfrak{g}}$, where
\[
\begin{split} &
\partial_{\mathfrak{g}}=\partial+d^{'1,0},\,\,\,
\bar{\partial}_{\mathfrak{g}}=\bar{\partial}+d^{'0,1}.
\end{split}
\]
It is straightforward to check that
\[
\bar{\partial}_{\mathfrak{g}}^2=0,\,\,\,
\partial_{\mathfrak{g}}^2=0,\,\,\,
\mathrm{and}\,\,\,
\partial_{\mathfrak{g}}\bar{\partial}_{\mathfrak{g}}+
\bar{\partial}_{\mathfrak{g}}\partial_{\mathfrak{g}}=0.
\]

\begin{defn}
The \emph{equivariant basic Dolbeault cohomology} of $M$, denoted by
$H_{\mathfrak{g}}^{p,*}(M, \mathcal{F})$,
is defined to be the cohomology of the differential complex $\{ \Omega_{\mathfrak{g}}^{p,*}(M,\mathcal{F}), \bar{\partial}_{\mathfrak{g}}\}$.
\end{defn}

A foliation analogue of Carrell--Lieberman type theorem was first established for structural Lie algebra actions on transversely K\"ahler Killing foliations in \cite[Theorem 7.5]{GNT16}. The following result is partly motivated by extending \cite[Theorem 7.5]{GNT16} to transversely K\"ahler foliations which may not be Killing, and to symmetries which may not come from the structural algebra.

\begin{thm}\label{thm2}
Let $G$ be a compact torus, and let $\mathcal{F}$ be a taut transversely K\"{a}hler foliation on a closed manifold $M$.
Suppose that there is a Hamiltonian action of $G$ on $(M,\mathcal{F})$ which is also holomorphic.
Then each connected component of the fixed-leaf set $X$ inherits a taut transversely K\"ahler foliation from that of $M$. Furthermore, $
H^{p,q}(M,\mathcal{F})=0$
for
$|p-q|>\, \mathrm{codim}_{\mathbb{C}}\,(\mathcal{F}\vert_{X})$,
where $\mathrm{codim}_{\mathbb{C}}\,(\mathcal{F}\vert_{X})$ is the maximum of the finite set \[\{\mathrm{codim}_{\mathbb{C}}\,(\mathcal{F}\vert_{X_i})\,\vert\, X_i \,\text{is a connected component of }\, X\}.\]\end{thm}

\begin{proof}
Equip the equivariant basic Cartan complex $\{\Omega_{\mathfrak{g}}(M,\mathcal{F}), d_{\mathfrak{g}}\}$
with the bi-grading given by
$$
\Omega^k_{\mathfrak{g}}(M,\mathcal{F})
=\bigoplus_{p+q=k}
\Omega^{p,q}_{\mathfrak{g}}(M,\mathcal{F}),
$$
together with the vertical differential
$\bar{\partial}_{\mathfrak{g}}$ and the horizontal differential $\partial_{\mathfrak{g}}$.
It was shown in \cite[Theorem 7.1]{GNT16} that the spectral sequence of the double complex
$\{\Omega^{\bullet,\bullet}_{\mathfrak{g}}(M,\mathcal{F}), \bar{\partial}_{\mathfrak{g}},\partial_{\mathfrak{g}}\}$ relative to the horizontal filtration degenerates at the $E_1$ term.

Since the connected components of fixed-leaf set $X$ is a $\mathcal{F}$-saturated closed submanifold of $M$, the normal bundle $TX/T\mathcal{F}\vert_X$ of the restricted foliation $\mathcal{F}\vert_X$ is a subbundle of the pullback of $TM/T\mathcal{F}$ to $X$. Let $\mathcal{J}:Q\rightarrow Q$ be the transversely holomorphic structure on $(M,\mathcal{F})$.
Applying \cite[Proposition 4.2]{LS18}, it is straightforward to check that $TX/T(\mathcal{F}\vert_X)$ is an invariant subbundle of $\mathcal{J}:Q\vert_X\rightarrow Q\vert_X$;
moreover, a routine check of definition shows that $\mathcal{J}:TX/T(\mathcal{F}\vert_X)\rightarrow TX/T(\mathcal{F}\vert_X)$ is also integrable. Hence $X$ inherits the structure of a transversely K\"ahler foliation from that of $M$.
 Moreover, by Proposition \ref{submfld} in the appendix the inherited transversely K\"ahler foliation on each connected component of $X$ is taut.

Observe that the induced transverse $G$-action on each $X_i$ is \emph{trivial}, which implies that $\Omega(X_i,\mathcal{F}|_{X_{i}})$ is a trivial $\mathfrak{g}^{\star}$-module.
Applying Theorem \ref{basic-ddbar-lemma} to the taut transversely K\"ahler foliation $(X_i, \mathcal{F}\vert_{X_i})$, an easy calculation shows that the spectral sequence associated to the double complex $\{\Omega^{*,*}_{\mathfrak{g}}(X_i, \mathcal{F}\vert_{X_i}), \bar{\partial}_{\mathfrak{g}}, \partial_{\mathfrak{g}}\}$
relative to the horizontal filtration also degenerates at the $E_1$ term.

By Theorem \ref{thm1}, the restriction homomorphism
$$
i^{*}: H_G(M,\mathcal{F})\longrightarrow H_G(X,\mathcal{F}\vert_X)
$$
is injective, so is the homomorphism
$$
i^*: E_{\infty}(\Omega_{\mathfrak{g}}(M,\mathcal{F}))
\longrightarrow E_{\infty}(\Omega_{\mathfrak{g}}(X,\mathcal{F}\vert_X)).
$$
It then follows from the degeneracies of the spectral sequences at the $E_{1}$ terms that the homomorphism
$$
i^*: E_{1}^{p,q}(\Omega_{\mathfrak{g}}(M,\mathcal{F}))
\longrightarrow E_{1}^{p,q}(\Omega_{\mathfrak{g}}(X,\mathcal{F}\vert_X))
$$ is also injective.
Since
$E_1^{p,q}(\Omega_{\mathfrak{g}}(M,\mathcal{F}))
=H^{p,q}_{\mathfrak{g}}(M,\mathcal{F})$
and
$
E_1^{p,q}(\Omega_{\mathfrak{g}}(X,\mathcal{F}\vert_X))
=H^{p,q}_{\mathfrak{g}}(X,\mathcal{F}\vert_X)
$,
we conclude that the restriction homomorphism
\begin{equation}\label{m-x-inj}
i^*:H^{p,q}_{\mathfrak{g}}(M,\mathcal{F})
\longrightarrow H^{p,q}_{\mathfrak{g}}(X,\mathcal{F}\vert_X)= \displaystyle \bigoplus_{0\leq j\leq \text{min}\{p,q\}} S^{j}(\mathfrak{g}^*)\otimes H^{p-j,q-j}(X,\mathcal{F}\vert_X)
\end{equation}
is injective.

By our assumption, for all $\vert p-q\vert >\mathrm{codim}_{\mathbb{C}}\,(\mathcal{F}\vert_X)$, we have that $\Omega^{p-j,q-j}(X,\mathcal{F}\vert_X)=0$, which implies that $H^{p-j,q-j}(X,\mathcal{F})=0$.
Thus by the injectivity of \eqref{m-x-inj} we must have that $H^{p,q}_{\mathfrak{g}}(M,\mathcal{F})=0$.
It follows from \cite[Theorem 7.1]{GNT16} again that $H^{p,q}(M,\mathcal{F})=0$.
This completes the proof.
\end{proof}

The following result is a direct consequence of Theorem \ref{thm2}.
\begin{cor}\label{isolated-fixed-leave}
Under the same assumptions as stated in Theorem \ref{thm2}, if $\mathcal{F}$ has only finitely many leaves that are invariant under the action of $G$, then
$
H^{p,q}(M,\mathcal{F})=0
$
for any $p\neq q$.
\end{cor}

\subsection{Basic Betti numbers}\label{basic-Betti}

In this subsection we apply Theorem \ref{thm1} to study the basic Betti numbers of a Hamiltonian transversely symplectic foliation.
Throughout this section we will make the following assumptions.

\begin{flushleft}
(\textbf{A1}) $\mathcal{F}$ is a transversely symplectic foliation of co-dimension $q$  on a compact manifold $M$ which satisfies the transverse hard Lefshetz property.
\end{flushleft}

\begin{flushleft}
(\textbf{A2}) There is a clean Hamiltonian action of a compact torus $T$ on $(M, \mathcal{F})$ with the moment map $\Phi:M\rightarrow\mathfrak{t}^{*}$,
which has only finitely many leaves invariant under the action of $T$.
\end{flushleft}

\begin{flushleft}
(\textbf{A3}) $\mathcal{F}$ is a taut Riemannian foliation on $M$.
\end{flushleft}

We will also need the following Morse inequality for basic cohomologies established in  \cite[Theorem A]{Al93}.

\begin{thm}\label{Morse-inequality}
Suppose that $f:M\rightarrow\mathbb{R}$ is a basic Morse--Bott function on a Riemannian foliation $(M, \mathcal{F})$ of co-dimension $q$, that the critical submanifolds of $f$ are isolated closed leaves, and that $M$ is compact.
Let $b_j(M,\mathcal{F})=
\mathrm{dim}_{\mathbb{R}}H^{j}(M,\mathcal{F})$, and let $v_j(f)$ be the number of critical submanifolds of $f$ which has Morse index $j$.
Then we have the inequalities
\[ b_{0}(M,\mathcal{F})\leq\nu_{0}(f),\]
\[b_{1}(M,\mathcal{F})-b_{0}(M,\mathcal{F})\leq\nu_{1}(f)-\nu_{0}(f),\]
\[b_{2}(M,\mathcal{F})-b_{1}(M,\mathcal{F})+b_{0}(M,\mathcal{F})\leq\nu_{2}(f)-\nu_{1}(f)+\nu_{0}(f),\]
etc. Moreover, we also have the following equality
\[\sum^{q}_{j=0}(-1)^{j}b_{j}(M,\mathcal{F})=\sum^{q}_{j=0}(-1)^{j}\nu_{j}(f).\]

\end{thm}

\begin{prop}\label{perfect-morse}
Let $f=\Phi^{\xi}$ for a generic element $\xi\in\mathfrak{t}$.
Under the assumptions (\textbf{A1}), (\textbf{A2}) and (\textbf{A3}), $f$ must be a perfect basic Morse--Bott function, that is, $b_j(M,\mathcal{F})=\nu_j(f)$, for all $1\leq j\leq q$.
\end{prop}

\begin{proof}
It suffices to show that $b_{2k-1}(M,\mathcal{F})=0$, for any $1\leq k\leq [\frac{q+1}{2}]$.
Let $X$ be the set of fixed leaves as before.
Note that the restriction homomorphism (\ref{restriction-homo}) maps $H^{\text{odd}}_{T}(M,\mathcal{F})$ into $H^{\text{odd}}_{T}(X,\mathcal{F})$, which by assumption (\textbf{A2}) must vanish.
It follows immediately from
Theorem \ref{thm1} that
$H^{\text{odd}}_{T}(M,\mathcal{F})=0$.
However, by Theorem \ref{eq-formality} we have that $H^{\text{odd}}_T(M,\mathcal{F})=S\mathfrak{t}^*\otimes H^{\text{odd}}(M,\mathcal{F})$.
Therefore all the odd dimensional basic Betti numbers must be zero, from which Proposition \ref{perfect-morse} follows.
\end{proof}

Proposition  \ref{perfect-morse} has the following easy consequence.

\begin{cor}\label{basic-euler}
Consider the clean Hamiltonian action of a compact torus $T$ on a transversely symplectic foliation $\mathcal{F}$.
Suppose that it satisfies the assumptions (\textbf{A1}), (\textbf{A2}) and (\textbf{A3}).
Then we have that
\begin{itemize}
\item[(i)] $b_{2i+1}(M,\mathcal{F})=0$;
\item[(ii)] $b_{2i}(M,\mathcal{F})$ coincides with the total number of vertices with index $2i$  in the moment polytope $\Delta$;
\item[(iii)] the basic Euler characteristic number $\chi(M,\mathcal{F})$, given by
\[\chi(M,\mathcal{F})=
\displaystyle\sum_{i=0}^{\mathrm{codim}(\mathcal{F})}
(-1)^{i}b_i(M,\mathcal{F}),\]
equals the total number of vertices in the moment polytope $\Delta$.
\end{itemize}
\end{cor}

\section{Main example: symplectic toric quasifolds}\label{quasifolds}

In this section, we apply the tools developed in Section \ref{general-theory} to determine the basic Hodge numbers and basic Betti numbers of a symplectic toric quasifold.

\subsection{Symplectic toric quasifolds I: basic Hodge numbers}\label{quasifold-1}

We begin with a quick review of the definition of a symplectic toric quasifold.
Let $V$ be an $m$ dimensional vector space,  and $\Delta\subset V^*$ an $m$ dimensional simple convex polytope given by
\[
\Delta=\bigcap^{d}_{j=1}
\bigl\{\mu\in V^{*}\,|\,
\langle\mu,X_{j}\rangle\geq\lambda_{j}\bigr\},\]
where $X_{j}\in V$ and $\lambda_{j}$'s are real numbers.
Here $\Delta$ being simple means that at any vertex there are exactly $m$ edge vectors.
Now consider the standard action of torus $T^d$ on $\mathbb{C}^d$ given by
$$
(e^{2\pi i\theta_1}, \cdots, e^{2\pi i\theta_d})\cdot(z_1,\cdots, z_d)=(e^{2\pi i\theta_1}z_1,\cdots, e^{2\pi i\theta_d}z_d).
$$
This action is Hamiltonian relative to the standard symplectic form
$\omega_0= \dfrac{\sqrt{-1}}{2}\sum_{j=1}^d dz_j\wedge d\overline{z}_j$ on $\mathbb{C}^d$,
with a moment map
\begin{equation}\label{moment-map}
\Phi:\mathbb{C}^d\longrightarrow (\mathbb{R}^d)^*, \,(z_1,\cdots, z_d)\longmapsto -\frac{1}{2}\displaystyle \sum_{j=1}^d \vert z_j\vert^2 e_j^*+\lambda,
\end{equation}
where $\lambda\in (\mathbb{R}^d)^*$ is a constant vector, and $e_1^*, \cdots, e_d^*$ are the dual basis of the standard
basis $e_1,\cdots, e_d$ in $\mathbb{R}^d$.
Let $N$ be an immersed subgroup of $T^d$ whose Lie algebra $\mathfrak{n}$ is the kernel of the surjective map given by
\begin{equation}\label{proj}
\pi:\mathbb{R}^d\longrightarrow \mathbb{R}^m, \,\,\,
e_j\longmapsto X_j.
\end{equation}
Then the induced action of $N$ on $\mathbb{C}^d$ is also Hamiltonian with a moment map
$\Phi_N: \mathbb{C}^d\rightarrow \mathfrak{n}^*$ given by the composition of the map (\ref{moment-map}) with the natural projection map $(\mathbb{R}^d)^*\rightarrow \mathfrak{n}^*$.

It is straightforward to check that $M=\Phi_N^{-1}(0)$ is compact. The quotient space $M/N$ is defined by Prato \cite{Pr01} to be a \emph{symplectic toric quasifold}. Note that the infinitesimal action of $\mathfrak{n}$ on $M$ is free
and generates a transversely symplectic foliation $\mathcal{F}$. Clearly, when the group $N$ is connected,
the leaf space of the transversely symplectic foliation $\mathcal{F}$ coincides with $M/N$; however, in general they are different from each other. To apply the machinery developed in the paper and derive some cohomological information on the
symplectic toric quasifold $M/N$, one possible approach is to apply \cite[Propsition 4.3]{BZ17} to obtain a representation of $M/N$ as a Riemannian foliation. Here we take another approach by regarding $M/N$ as a diffeological space with a quotient diffeology inherited from that of $M$. By Corollary \ref{main-iso2},  the diffeological De Rham cohomology of $M/N$ is naturally isomorphic to $H(M, \mathcal{F})$. Therefore without loss of generality, from now on we will assume that $N$ is connected and $M/N$ coincides with $M/\mathcal{F}$. We refer the interested reader to \cite{I13} for a detailed exposition on diffeology.  We also note that a more recent work \cite{IP21} has shown that every quasi-fold has a representation as a diffeological space.

Clearly, the induced action of $T^d$ on the transversely symplectic foliation $(M,\mathcal{F})$ is Hamiltonian.
Indeed, the restriction of $\Phi$ to $M$ is a moment map for the induced action of $T^d$ on $(M,\mathcal{F})$, which by abuse of notation we also denote by $\Phi$.
Let
$\pi^*: (\mathbb{R}^m)^*\rightarrow (\mathbb{R}^d)^*$ be the dual of the map (\ref{proj}).
Then the image of $M$ under the moment map $\Phi$ is given by $\pi^*(\Delta)$.

\begin{lem} The foliation $\mathcal{F}$ is transversely K\"ahler and Killing. As a result, $\mathcal{F}$ must be taut.
\end{lem}

\begin{proof} The assertion that $\mathcal{F}$ is transversely K\"ahler follows from \cite[Proposition 6.2]{Lin18}.
The fact that $\mathcal{F}$ is a Killing foliation follows from \cite[Example 4.3]{GT10}. So it follows from Theorem \ref{molino-sheaf} that $\mathcal{F}$ is also taut.
\end{proof}

In view of Corollary \ref{main-iso2}, we will simply define the Hodge number $h^{p, q}$ of the symplectic toric quasifold $M/N$ to be $h^{p,q}(M, \mathcal{F})$.  An easy application of the presymplectic slice theorem \cite[Theorem C.5.1]{LS17} shows that the Hamiltonian $T^d$-action on
$(M,\mathcal{F})$ has only finitely many fixed leaves.
As an immediate consequence of Corollary \ref{isolated-fixed-leave}  we have the following result.
\begin{cor} \label{vanishing-hodge-number}
Let $(M,\mathcal{F})$ be the transversely K\"ahler foliation constructed out of a non-rational simple polytope $\Delta$ as above.
Then we have that  $h^{p,q}(M,\mathcal{F})=0$ when $p\neq q$.
\end{cor}

\begin{rem}
Corollary \ref{vanishing-hodge-number} provides a positive answer to a conjecture raised by Battagalia and Zaffran in \cite[Page 11812]{BZ15}. We refer to \cite[Section 6]{Lin18} for a different approach outlined there using the compatibility of the Kirwan map with the Hodge structures.
\end{rem}

\subsection{Symplectic toric quasifolds II: Betti numbers of diffeological De Rham cohomologies}

We will keep the same assumptions and notations as in Section \ref{quasifold-1}.
It is straightforward to check that the action of $T$ on $(M,\mathcal{F})$ is clean; moreover, the null subgroup is precisely the connected immersed normal subgroup $N$ as constructed in Section \ref{quasifold-1}.
The following result is a refinement of Corollary \ref{basic-euler} in the case of symplectic toric quasifolds.

\begin{thm}\label{Betti-numbers}
Let $M/N$ be the symplectic toric quasifold constructed in Section \ref{quasifold-1}, and let $b_i$ be the Betti number of its $i$-th diffeological De Rham cohomology. Then we have that
\begin{itemize}
  \item [(i)] $b_{2k+1}=0$;
  \item [(ii)] $b_{2k}=h_k$, where $(h_1,\cdots, h_m)$ is the $h$-vector of the $m$-dimensional simple polytope $\Delta$.
\end{itemize}
\end{thm}

\begin{proof} Assertion (i) is an immediate consequence of Proposition \ref{perfect-morse}.
To show Assertion (ii), we first note that since $\Delta$ is a simple polytope, its $h$-vector is given by
\begin{equation}\label{h-vector}
h_k=\displaystyle\sum_{i=0}^k (-1)^{k-i}\left(\begin{matrix}  m-i\\ m-k\end{matrix}\right)f_{m-i},
\end{equation}
where $f_j$ is the number of $j$-dimensional faces in $\Delta$ (cf. \cite{Ba11})
\footnote{ The definition used here is the dual version of that for a simplicial  polytope, see \cite[Definition 8.18]{Zi95}.}.
Fix a generic element $\xi\in \mathfrak{t}$.
We divide the rest of the proof into two steps.

\paragraph{\textbf{Step 1}}
Let $F\subset \Delta$ be a non-empty $(m-i)$-dimensional face, and let $\lambda_F$ be the unique vertex of $F$ with the minimal index.
We will show that
$\mathrm{ind}_{\xi}(\lambda_F)\leq 2i$.
Choose $x \in \Phi^{-1}(\lambda_F)$. Note that the leaf though $x$ is isolated and can be naturally identified with $T^d/H$, where $H$ is the isotropy subgroup of $T^d$ at $x$.
By the presymplectic slice theorem \cite[Theorem C.5.1]{LS17}, we may assume that $x$ has a $T^d$-invariant open neighborhood in $M$ given by $\mathfrak{M}= T^d\times^H S$,  where $S=T_xM/T_x\mathcal{F}$ is a symplectic vector space.
Since all fixed leaves are isolated in our case, a simple dimension count shows that $\text{dim}(S)=2\text{dim}(H)=2m$.
Moreover, let $\mathfrak{h}=\text{Lie}(H)$,  let $\mathfrak{t}=\text{Lie}(T^d)$, and identify $\mathfrak{t}^*$ with $(\mathfrak{t}/\mathfrak{h})^*\times \mathfrak{h}^*$. Then the restriction of $\Phi$ to $\mathfrak{M}$ is given by
\[
\Phi([g, v])=\lambda -\frac{1}{2}\displaystyle \sum_{j=1}^m (x_j^2+y_j^2)\cdot w_j,\,\,\, \mathrm{for\,\,\,all}\,\,g\in T^d\,\,\,
\mathrm{and}\,\,v\in S,
\]
where $(x_1,y_1,\cdots,x_m,y_m)$ are Darbourx coordinates of $v$ on $S$, and $w_1,\cdots, w_m \in \mathfrak{h}^*\subset\mathfrak{t}^*\cong (\mathbb{R}^d)^*
$ are weight vectors.
It follows that there is a neighborhood $U$ of the vertex $\lambda_F\in (\mathbb{R}^d)^*$ such that
$U\cap \Delta$ is the convex cone with vertex $\lambda_F$ and edge vectors $w_1,\cdots, w_m$, and such that
\begin{equation}\label{morse-lemma}
\Phi^{\xi}([g, v])=\langle\lambda, \xi\rangle-\frac{1}{2}\cdot\displaystyle
\sum_{j=1}^m\langle w_j,\xi\rangle\cdot(x_j^2+y_j^2).
\end{equation}
Since $x$ is a local minimum of the restriction of $\Phi^{\xi}$ to $\Phi^{-1}(F)$, and since $F$ is $(m-i)$-dimensional, (\ref{morse-lemma}) implies that $\mathrm{ind}_{\xi}(\lambda_F)$ is at most $2i$.

\paragraph{\textbf{Step 2}}
We deduce that $b_{2k}(M,\mathcal{F})=h_k$ from (\ref{h-vector}).
In what follows, for a finite set $E$ we will denote by $\bigl\vert E\bigr\vert$ its cardinality.
For all $0\leq k\leq m$, let
\[
\begin{split}&
A_k=\bigl\{ F \,\vert\, F \,\text{is an}\, (m-k)-\text{\,dimensional face of }\,\Delta\,\text{ such that}\, \text{ind}_{\xi}(\lambda_F)=2k\bigr\},\\
&B_k=\bigl\{ F \,\vert\, F \,\text{is an}\, (m-k)-\text{\,dimensional face of }\,\Delta\,\text{ such that}\, \text{ind}_{\xi}(\lambda_F)<2k\bigr\}.
\end{split}\]
By definition, we have that $f_{m-k}=\bigl\vert A_k\bigr\vert+\bigl\vert B_k\bigr\vert$.
Let $F_1$ and $F_2$ be two different faces in $A_k$. Since $\Delta$ is simple, it follows easily from (\ref{morse-lemma}) that $\lambda_{F_1}$ and $\lambda_{F_2}$ must be two different vertices of $\Delta$ with index $2k$. Conversely, for any vertex
$\nu\in \Delta$ with index $2k$, there is an $(m-k)$-dimensional face $F$ of $\Delta$ such that $\nu=\lambda_F$.
So by Corollary \ref{basic-euler}
\begin{equation}\label{step1}
b_k(M,\mathcal{F})=\bigl\vert A_k\bigr\vert.
\end{equation}
In order to calculate $\bigl\vert B_k\bigr\vert$, we introduce a partition of the set $B_k$ as follows.
Let $\{\sigma_1,\cdots,\sigma_p\}$ be the collection of all $(m-k+1)$-dimensional faces of $\Delta$.
Define
\[
B_k^j=\bigl\{ F\in B_k\,\vert\, F\subset \sigma_j,\,\, \lambda_F=\lambda_{\sigma_j}\bigr\},\,\,
\mathrm{for\,\,\,all}\,\, 1\leq j\leq p.
\]
By definition we have that $B_k=\bigcup_{j=1}^p B_k^j$, and that
\begin{equation}\label{intersection} B_k^{i_1}\cap\cdots\cap B_k^{i_s}=
\bigl\{ F\in B_k\,\vert\, F\subset \sigma_{i_j},\,\, \lambda_F=\lambda_{\sigma_{i_j}},\,\,
\mathrm{for\,\,\,all}\,\,\, 1\leq j\leq s\bigr\},
\end{equation}
where $1\leq i_1< \cdots<i_s\leq p$.
Since the moment polytope $\Delta$ is simple, for any sequence of integers $1\leq i_1< \cdots<i_s\leq p$, the set (\ref{intersection}) is either empty or has exactly one element.
If (\ref{intersection}) is non-empty, then there exists a unique $(m-k+s)$-dimensional face $\tau$, such that for the unique face $F\in B_k^{i_1}\cap\cdots\cap B_k^{i_s}$, $\lambda_F=\lambda_{\tau}$.
Conversely, let $\tau$ be an $(m-k+s)$-dimensional face of $\Delta$.
Then $\tau$ has $m-k+s$ edge vectors at its vertex with the  minimal index $\lambda_{\tau}$, any choice of $m-k$ edge vectors from which uniquely determines a face $F$ and a sequence of integers
$1\leq i_1\cdots<i_s\leq k$ such that
$F\in B_k^{i_1}\cap\cdots\cap B_k^{i_s}$.
It follows that
\[
\displaystyle\sum_{1\leq i_1<\cdots <i_s\leq k}\bigl\vert B_k^{i_1}\cap\cdots\cap B_k^{i_s}\bigr\vert
=\binom{ m-k+s}{m-k} f_{m-k+s}.
\]
Therefore by the \emph{Inclusion-Exclusion Principle} in elementary combinatorics
\begin{eqnarray}\label{step2}
  \bigl\vert B_k\bigr\vert
  &=&
  \sum_{i=1}^p\bigl\vert B_k^i\bigr\vert+\cdots +(-1)^{s-1}\sum_{1\leq i_1<\cdots <i_s\leq p}\bigl\vert B_k^{i_1}\cap \cdots \cap B_k^{i_s}\bigr\vert \nonumber\\
  & & +\cdots+(-1)^{k-1}\sum_{1\leq i_1<\cdots <i_k\leq p}\bigl\vert B_k^{i_1}\cap\cdots\cap B_k^{i_k}\bigr\vert \nonumber\\
  &=&
  \displaystyle\sum_{s=1}^{k}(-1)^{s-1}\binom{ m-k+s}{m-k} f_{m-k+s}.
\end{eqnarray}

Combining (\ref{h-vector}), (\ref{step1}), and (\ref{step2}) we conclude that $b_{2k}(M,\mathcal{F})=h_k$. This completes the proof of Theorem \ref{Betti-numbers}.
\end{proof}


\appendix \section{Transversely K\"ahler submanfiolds}

In this appendix we apply the transverse integration introduced by
Sergiescu to prove a foliation analogue of the classical Wirtinger theorem in K\"ahler geometry. We refer to \cite{Mo88}, \cite{Ser85} and \cite{LS18} for a detailed account on transverse integration and Molino's structure theory of Riemannian foliations. In particular, we need the following characterization
by Sergiescu on when
a Riemannian foliation is taut.

\begin{thm} \label{molino-sheaf}(\cite{Ser85})
Let $\mathcal{F}$ be a Riemannian foliation on a compact manifold.
Then $\mathcal{F}$ is taut if and only if the top wedge power of its Molino sheaf has a nowhere vanishing global section.
\end{thm}

\begin{prop} \label{submfld}
Suppose that $\mathcal{F}$ is a transversely K\"ahler foliation on a compact manifold $M$ equipped with a transversely K\"ahler form $\omega$ and a transversely K\"ahler metric $g$, that $\mathcal{F}$ is taut, and that $ X$ is a closed saturated submanifold of $M$ such that $(\omega\vert_X, g\vert_X)$ is transversely K\"ahler with respect to $\mathcal{F}\vert_X$.   Let
the co-dimensions of $\mathcal{F}$ in $M$ and $\mathcal{F}\vert_X$ in $X$ be $2n$ and $2k$ respectively. Then
 $[\omega^k\vert_X]$ defines a non-trivial basic cohomology class in $H^{2k}(X, \mathcal{F}\vert_X)$. In particular, $(X, \mathcal{F}\vert_X)$ is also taut.
\end{prop}

\begin{proof} Let $\pi: P\rightarrow M$ be the corresponding transverse orthonormal frame bundle on $M$ associated to the transversely K\"ahler metric $g$, which has a structure group $SO(2n)$.Then the foliation $\mathcal{F}$ naturally lifts to a transversely parallizable foliation $\mathcal{F}_P$ on $P$.  As a result there is a locally trivial fibration $\rho: P\rightarrow W$ whose fibers are leaf closures of $\mathcal{F}_P$.

A transverse Riemnannian metric on $P$ can be described as follows. In our situation the transverse Levi-Civitta connection $\theta_{LC}$ defines a basic one form on $P$ that takes values in the Lie algebra
 $\mathfrak{t}=so(2n)$. Choose an invariant inner product on $\mathfrak{t}$. Let
$f_1, \cdots, f_r$ be an oriented orthonormal basis in the dual of $\mathfrak{t}$, and let $\alpha_i= \theta^*_{LC}(f_i)$, $\forall\, 1\leq i \leq r$.  Then $\alpha_1, \cdots, \alpha_r$ are $\mathcal{F}_P$ basic $1$-forms such that
$ g_P=\pi^* g+ \displaystyle \sum_{i=1}^r\alpha_i \otimes \alpha_i$
is a transverse Riemnannian metric on $P$.

Clearly,  $X_P:=\pi^{-1}(X)$ is a closed saturated submanifold of $P$,  and carries a transverse Riemannian metric $g_{X_P}$ which is the restriction of $g_P$ to $X_P$. Note that by Lemma \ref{kahler-form}, $\frac{1}{k!} \omega^k\vert_X$ is a transverse volume form on $X$ associated to the transverse metric $g\vert_X$. It follows that when restricting to $X_P$, $\frac{1}{k!} \pi^*\omega^k \wedge \alpha_1\wedge\cdots \wedge \alpha_r$ is a transverse volume form associated to the transverse Riemannian metric $g_{X_P}$. In particular, it is a no-where vanishing basic form of top degree. Now that $M$ is taut, Theorem \ref{molino-sheaf} implies that there is a no-where vanishing multi-vector $ s= v_1\wedge \cdots \wedge v_l$ on $P$ representing a global section of the top wedge power of the Molino sheaf of $\mathcal{F}_P$. Here $v_1, \cdots, v_l$ are local transverse vector fields tangent to leaf closures,  and $l=\text{dim}(\overline{\mathcal{F}}_P) -\text{dim}(\mathcal{F})$.  Since $X_P$ is a closed saturated submanifold, at any point $z\in X_P$,
$v_1, \cdots, v_l$ are also tangent to $X_P$.  For simplicity let us still write $s$ for the restriction of $s$ to $X_P$, and similarly for $\pi^*\omega^k$ and $\alpha_i$'s. Since basic forms are invariant under the action of Molino's centralizer sheaf, on $X_P$ the interior product
\begin{equation}\label{interior-product} \iota_s( \pi^*\omega^k \wedge \alpha_1\wedge \cdots \wedge \alpha_r)\end{equation} is basic with respect to leaf closures in $X_P$, and thus descends to a form $\gamma$  of top degree on the submanifold $X_W=\pi_W(X_P)$ of $W$. By definition, the transverse integration of $\omega^k$ over $X$ is
\[ \fint_{X/\mathcal{F}}\, \omega^k:=\int_{X_W}\, \gamma.\]
However it is clear that (\ref{interior-product}) is no-where vanishing on $X_P$. This implies that $\gamma$ is a volume form on $X_W$ and $\fint_{X/\mathcal{F}}\, \omega^k\neq 0$. So by the Stokes' formula for transverse integration $\omega^k$ defines a non-trivial basic cohomology class in $H^{2k}(X, \mathcal{F}\vert_X)$, see \cite[Proposition 6.2.7]{LS18}.
\end{proof}

\section{ De Rham cohomology of a diffeological space}\label{diffeology}

In this appendix, we prove that the diffeological cohomology of the symplectic toric quasifold $M/N$, the construction of which is reviewed in Section \ref{quasifolds}, is naturally isomorphic to the basic cohomology of the transversely symplectic foliation $(M,\mathcal{F})$, where $\mathcal{F}$ is the foliation induced by the action of the Lie algebra of $N$ on $M$.
It is noteworthy though that technically the main results established in Appendix \ref{diffeology} apply in a broader context.
We refer to \cite{Pr01}, \cite{I13}, and \cite{IP21} for a detailed account of background materials on quasifolds and diffeological spaces.

\begin{defn}
Let $X$ be a set.
A \emph{plot} on $X$ is a map $\alpha: U \rightarrow X$ defined on an open subset $U$ of a Euclidean space $\mathbb{R}^n$, $n\geq 0$.
A \emph{diffeology} $\mathcal{D}$ on $X$ is a family of plots on $X$ satisfying the following axioms:
\begin{itemize}
\item[(1)] All constant plots on $X$ are in $\mathcal{D}$.
\item[(2)] Suppose that $V$ and $U$ are open subsets of $\mathbb{R}^m$ and $\mathbb{R}^n$
           respectively, that
           $\alpha$ is a plot in $\mathcal{D}$ defined on $U$, and that $h: V\rightarrow U$ is a smooth map. Then $\alpha\circ h \in \mathcal{D}$.
\item[(3)] Suppose that $\alpha: U\rightarrow X$ is a plot, and that $\forall\, p\in U$, there
           exists an open neighborhood $V_p$ of $p$ contained in $U$, such that $\alpha\vert_{V_p}\in \mathcal{D}$. Then $\alpha \in \mathcal{D}$.
\end{itemize}
If $\mathcal{D}$ is a diffeology on a set $M$, then $(M, \mathcal{D})$ is called a \emph{diffeological space}.
\end{defn}

Let $(M, \mathcal{D})$ be a diffeological space.
We say that a family of plots $\mathcal{C}$ on $M$ is a \emph{generating family}, if any plot $f: U\rightarrow M$ in $ \mathcal{D}$  is either a constant plot, or for each $ x\in U$, there exist a neighborhood $V_x$ of $x$, a plot $g: W\rightarrow M$ in $\mathcal{C}$,  and a smooth map $h: V_x\rightarrow W$, such that $f\vert_{V_x}= g \circ h$.

\begin{ex} (Manifold diffeologies)
Let $\{ (V_i, \phi_i)\}_{i\in I}$ be an atlas on a differentiable manifold $M$.
Then the finest diffeology on $M$ that contains $\mathcal{C}=\{\phi_i^{-1}\}_{i\in I}$ is called the manifold diffeology of $M$.
It is easy to check that $\mathcal{C}$ is a generating family.
\end{ex}

\begin{defn} A map $f: (X, \mathcal{D}) \rightarrow (Y, \mathcal{D}')$ between two diffeological spaces is said to be \emph{smooth}, if
$f\circ \alpha \in \mathcal{D}'$, for any $\alpha\in \mathcal{D}$.
\end{defn}

\begin{defn} \label{forms} Let $(X, \mathcal{D})$ be a diffeological space.
A \emph{differential form} $\gamma$ on $X$ is a map which assigns to each plot $\alpha: U\rightarrow X$ in $\mathcal{D}$ a differential form  $\gamma_{\alpha}$ on $U$, in such a way that satisfies the following compatibility condition
\begin{equation}\label{compatibility} \gamma_{\alpha\circ f}=f^*\gamma_{\alpha},\end{equation}
where $f: \mathbb{R}^m \supset V \rightarrow U$ is a smooth map.
We denote by $\Omega(X, \mathcal{D})$ the space of differential forms on the diffeological space $(X, \mathcal{D})$.
The exterior differential
$d:\Omega^*(X, \mathcal{D})\rightarrow \Omega^{*+1}(X, \mathcal{D})$ is given by the formula $\{\gamma_{\alpha}\} \mapsto \{ d\gamma_{\alpha}\}$.
\end{defn}

Clearly, $d^2 =0$.  We define the De Rham cohomology of the diffeological space $(X, \mathcal{D})$ to be
$ H(X, \mathcal{D})=\text{ker}(d)/\text{im}(d)$.
Now suppose that $f: (X, \mathcal{D})\rightarrow (Y, \mathcal{D}')$ is a smooth map between two diffeological spaces.
Then we have a natural pullback map $f^*:\Omega(Y, \mathcal{D}')\rightarrow \Omega(X, \mathcal{D})$ given by
\[
(f^*\gamma)_{\alpha}:=\gamma_{f\circ \alpha},\,\,\,\forall\, \alpha \in \mathcal{D}.
\]

\begin{defn} \label{quotient}
Let $(X, \mathcal{D})$ be a diffeological space, and let $\sim$ be an equivalence relation on $X$, and let $\pi: X\rightarrow X/\sim$ be the quotient map.
Consider a family of plots on $X/\sim$, denoted by $\mathcal{D}'$,  characterized by the following property:  a plot $\gamma  : U\rightarrow X/\sim$ lies in $\mathcal{D}'$  if and only if either $\gamma$ is a constant plot, or $  \forall\, x\in U$,  there exists a neighborhood $V $ of $x$, and a plot $\alpha \in \mathcal{D}$ defined on $V$, such that
\begin{equation}\label{quotient-plot}
\gamma \vert_V =\pi\circ \alpha.
\end{equation}
Then $\mathcal{D}'$ is a diffeology on $X/\sim$, called the \emph{quotient diffeology}.
It is the finest diffeology on $X/\sim$ such that the quotient map $\pi$ is smooth.
\end{defn}

Now let $(M, \mathcal{F})$ be a foliation of co-dimension $q$, and let $\pi: M\rightarrow M/\mathcal{F}$ be the quotient map from $M$ onto its leaf space.
From now on, we will regard $M/\mathcal{F}$ as a diffeological space with the quotient diffeology $\mathcal{D}'$ induced from the quotient map and the manifold diffeology on $M$.
Following \cite{HMS11}, we explain there is a natural isomorphism $\pi^* :\Omega(M/\mathcal{F},\mathcal{D}')\rightarrow \Omega(M, \mathcal{F})$.
However, the description of the isomorphism we present here is perhaps more explicit, as we are able to avoid the application of pseudogroups by capitalizing on the following definition of a foliation.
By a co-dimension $q$ foliation $\mathcal{F}$ we mean a maximal family of submersions $f_{i}: U_i\rightarrow \mathbb{R}^q$,  where $\{U_i\}_{i\in I}$ is an open cover of $M$ and where for $i, j\in I$, and for $x \in U_i\cap U_j$, there exists a local diffeomorphism $\phi_{ji}^x$ of $\mathbb{R}^q$ such that $f_j =\phi_{ji}^x \circ f_i$ in a neighborhood $V_x$ of $x$.
Suppose that $J$ is a subset of  $I$, such that for each $i\in J$, the map $f_i: U_i\rightarrow f(U_i)\subset\mathbb{R}^q$ admits a section $s_i: f(U_i)\rightarrow U_i\subset M$, where $f(U_i)$ is open in $\mathbb{R}^q$.
A crucial observation is that
\[
\mathcal{C}:=\{q_i:= \pi \circ s_i
: f_i(U_i)\rightarrow M/\mathcal{F}\}_{i\in J}
\]
is a generating family of the quotient diffeology $\mathcal{D}'$.

Clearly, we have the following commutative diagram.
\begin{equation}\label{f-chart}
\vcenter{
 \xymatrix@=1cm{
  U_i \ar[dr]_{\pi} \ar[r]^{f_i}
                & f_i(U_i)
                \ar[d]^{q_i}  \\
                & M/ \mathcal{F}.             }}
\end{equation}
Note that if there is another map $q_i'$ which satisfies $q_i'\circ f_i=q_i\circ f_i$ on an open subset $W$ of $U_i$, then $q_i=q_i'$ on $f_i(W)$.
In particular, since $q_j\circ f_j=(q_j\circ \phi_{ji}^x) \circ f_i=q_i\circ f_i$ holds on $U_i\cap U_j\neq \emptyset$, we must have that $q_i=q_j\circ \phi_{ji}^x$.
It follows from (\ref{compatibility}) that for a diffeologocal form $\alpha$ on $M/\mathcal{F}$, $\alpha_{q_i}=(\phi_{ji}^x)^*\alpha_{q_j}$, and thus $$
f_i^*\alpha_{q_i}=f_i^* (\phi_{ji}^x)^*\alpha_{q_j}
=(\phi^x_{ji}\circ f_i)^*\alpha_{q_j}=f_j^*\alpha_{q_j}.
$$
Therefore the family of locally defined basic forms $\{f_i^*\alpha_{q_i}\}_{i\in J}$ pieces together to give a globally defined basic form $\pi^*\alpha \in \Omega(M, \mathcal{F})$.
If $\pi^*\alpha=0$, then it follows easily from the construction that $\alpha_{q_i}=0$ for each plot $q_i\in \mathcal{C}$. Since $\mathcal{C}$ is a generating family,
as an easy consequence of (\ref{compatibility}) we see that $\alpha_{f}=0$ for each plot $f$ in $\mathcal{D}'$.
Thus the morphism $\pi^*$ we constructed is injective.

To show that $\pi^*$ is surjective,  we need to show given a basic form $\gamma$ on $M$, there is a diffeological form $\alpha$ such that $\pi^*\alpha=\gamma$.
For a plot $q_i \in \mathcal{C}$, it is easy to see that $\gamma\vert_{U_i}$ descends to a form on its domain $f_i(U_i)$, which we define to be $\alpha_{q_i}$. For an arbitrary plot $f: O\rightarrow M/\mathcal{F}$ that lies in $ \mathcal{D}'$, if it is constant, simply define $\alpha_f=0$. Otherwise we define $\alpha_f$ as follows.
Since $\mathcal{C}$ is a generating family, for $x\in
O$, there exists an open neighborhood $V_x$ of $x$, a plot $q_i\in \mathcal{C}$, and a smooth map $h_i: V_x\rightarrow f_i(U_i)$, such that $f\vert_{V_x}=q_i\circ h_i$. So there is a family of locally defined forms $\{ h_i^*\alpha_{q_i}\}$.
We claim any two such forms would agree with each other on the overlap of their domain.
To see this, it suffices to show that if $h_i:V\rightarrow f_i(U_i)$ and $h_j: V\rightarrow f_j(U_j)$ are two smooth maps such that $q_i\circ h_i=q_j\circ h_j$, then $h_i^*\alpha_{q_i}=h_j^*\alpha_{q_j}$.
First identify $f_i(U_i)$ and $f_j(U_j)$ with transversals in $M$ using sections $s_i$ and $s_j$ respectively.
Then by applying \cite[Lemma 3.3]{HMS11} to the current situation, we see there is a countable family of open subsets $W_s$ in $V$, such that
$\bigcup W_s$ is dense in $V$, and such that on each $W_s$, $h_j=\phi^s\circ h_i$, where $\phi^s$ is a holonomy diffeomophism from an open subset in $ f_i(U_i)$ onto an open subset in $ f_j(U_j)$. Since $\alpha_{q_i}$ and $\alpha_{q_j}$ are induced by basic forms, we get that $h_j^* \alpha_{q_j}=h_i^*(\phi^s)^*\alpha_{q_j}=h_i^*\alpha_{q_i}$ holds on each $W_s$, and so holds on $V$ by a continuity argument.
Therefore we have a well-defined form $\alpha_{f}$ for each plot $f$ in $\mathcal{D}'$. It is straightforward to check that this procedure defines a diffeological form $\alpha$ on $M/\mathcal{F}$ such that $\pi^*\alpha=\gamma$.
This completes the proof of the following result.

\begin{thm}(\cite{HMS11})\label{HMS}
The quotient map $\pi: M\rightarrow M/\mathcal{F}$ induces a well-defined isomorphism of De Rham complexes
\begin{equation}\label{iso1}
\pi^*:( \Omega(M/\mathcal{F}, \mathcal{D}'), d)\longrightarrow ( \Omega(M, \mathcal{F}), d),
\end{equation}
and therefore an isomorphism
$ H(M/\mathcal{F}, \mathcal{D}')\overset{\pi^*}{\cong} H(M, \mathcal{F})$.
\end{thm}

We say that a group $H$ acts on a diffeological space $(X, \mathcal{D})$, if $H$ acts on $X$ as a set, and if for every $h\in H$, the action map
\begin{equation}\label{action}
\phi_{h}:X\longrightarrow X,\,\,\, x\longmapsto h\cdot x
\end{equation}
is smooth.
Note that the action of $H$ on $X$ induces an action of $H$ on $\Omega(X, \mathcal{D})$ as follows:
for any $\alpha\in \Omega(X, \mathcal{D})$, $h\in H$, and $f: U\rightarrow X$ in $\mathcal{D}$, define the differential form $h^*\alpha$ by setting $(h^*\alpha)_f=\alpha_{\phi_{h}\circ f}$.
Let $\pi_H: (X, \mathcal{D}) \rightarrow (X/H, \mathcal{D}_H)$ be the quotient map from
$X$ onto the quotient diffeological space $X/H$, and let $\Omega(X, \mathcal{D})^H$ be the space of diffeological forms on $X$ that is invariant under the action of $H$.
Applying Definition \ref{forms}, it can be shown that for any diffeological form $\alpha\in \Omega(X/H, \mathcal{D}_H)$, its pullback $\pi^*\alpha$ is invariant under the $H$-action.
Thus we have a natural pullback map
\begin{equation}\label{pullback}
\pi^*_H: \Omega(X/H, \mathcal{D}_H)\longrightarrow \Omega(X, \mathcal{D})^H.
\end{equation}
Indeed, we have the following result.

\begin{prop} \label{d-quotient} The pullback map  (\ref{pullback}) is injective.
Moreover, if $G$ is countable, and if there is a topology on $X$ such that the for each $h\in H$, the action map (\ref{action}) is continuous, then (\ref{pullback}) is an isomorphism.

\end{prop}
\begin{proof} The first assertion is an easy consequence of Definition \ref{quotient}.
The second assertion can be shown using the same argument given in the proof of \cite[Lemma 3.3]{HMS11}.
\end{proof}

Finally we will prove a refinement of Theorem \ref{HMS} in the presence of a group action. Suppose that a Lie group $G$ acts freely on a differentiable manifold $M$.
Since the action is free, the Lie algebra action of $\mathfrak{g}:=\text{Lie}(G)$ gives rise to a foliation $\mathcal{F}$, on which the elements of $G$ act as foliation preserving diffeomorphisms.  We will denote by $\Omega_b(M)$ the space of $G$-basic forms on $M$.
Clearly, when $G$ is connected, $\Omega_b(M)=\Omega(M, \mathcal{F})$. However, when $G$ is disconnected, $\Omega(M, \mathcal{F})$ can be a proper subspace of $\Omega_b(M)$.

Let $G_0$ be the connected component of $G$ that contains the identity element.
It is a normal subgroup of $G$; moreover, the quotient group $H:=G/G_0$ is countable.
Note that the action of $G$ on $M$ naturally induces an $H$ action on the leave space $M/\mathcal{F}$, which is continuous with respect to the usual quotient topology.
It is easy to check that the $H$ action also preserves the quotient diffeology on $M/\mathcal{F}$.  A straightforward check of definitions yields the following result.

\begin{cor} \label{eq-variance}
Suppose that $\alpha \in \Omega(M/\mathcal{F}, \mathcal{D}')$, that $g\in G$, and that $h=[g]\in H$ is the image of $g$ under the quotient map $G\rightarrow G/G_0$.
Then we have that
$g^*\pi^*\alpha = \pi^* h^*\alpha$.
As a result, we have a natural isomorphism
\[ \pi^*: \Omega(M/\mathcal{F}, \mathcal{D}')^H\longrightarrow \Omega(M, \mathcal{F})^G.\]
\end{cor}

Denote by $\mathcal{D}_G$ the quotient diffeology on the orbit space $M/G$ that inherits from the quotient map $\pi_G: M\rightarrow M/G$.
Note that there is a natural diffeological diffeomorphism from $((M/\mathcal{F})/H, \mathcal{D}_H)$ to $(M/G, \mathcal{D}_G)$.
Combining Corollary \ref{eq-variance} with Proposition \ref{d-quotient}, we have the following result.

\begin{thm} \label{main-iso}
There is a natural isomorphism of De Rham complexes from $(\Omega(M/G, \mathcal{D}), d)$ onto $(\Omega(M, \mathcal{F})^G, d)$.
\end{thm}

Now suppose that $G$ is a subgroup of a compact Lie group $K$, and that the foliate action of $G$ on $M$ extends to a foliate action of $K$.
Clearly, $\bar{G}$, the closure of $G$, would have to be compact as well; moreover, we have that $(\Omega(M, \mathcal{F})^G=(\Omega(M, \mathcal{F})^{\bar{G}}$.
Using the usual averaging technique, we obtain the following result.
\begin{cor} \label{main-iso2}
Under the above assumptions, there exists a natural isomorphism from  $H^{\ast}(M/G, \mathcal{D}_G)$ to $H^{\ast}(M,\mathcal{F})$.
\end{cor}
\subsection*{Funding}
The second author is partially supported by the National Natural Science Foundation of China (grant number 11701051).

\subsection*{Acknowledgements}
Y. Lin would like to thank the School of Mathematics of Sichuan University for hosting his research visit during the summer of 2017 and 2018 when he was working on things related to this project.
X. Yang would like to thank the Department of Mathematics of Cornell University for providing him an excellent working environment since he joined Cornell University as a visiting scholar in September 2017.
He is also indebted to the China Scholarship Council for the financial support during his visit.
Both authors are grateful to Professor Reyer Sjamaar for many useful discussions.
The authors thank the anonymous referee for many valuable comments and suggestions.


\end{document}